\documentclass[12pt,letterpaper,reqno]{amsart}

\usepackage{amsthm}
\usepackage{mathrsfs}
\usepackage{amsfonts,amssymb,amsmath}
\input amssym.def 
\input amssym.tex

\usepackage[dvips]{graphicx}
\usepackage{hyperref}
\usepackage{thmtools}

\declaretheoremstyle[
bodyfont=\normalfont,
]{remstyle}

\newtheorem{lemma}{\bf Lemma}

\newtheorem{theorem}{\bf Theorem}

\newtheorem{corr}{\bf Corollary}

\newtheorem{remark}{\bf Remark}

\newcommand\be{\begin{eqnarray*}}
\newcommand\ee{\end{eqnarray*}}
\newcommand\beq{\begin{equation}}
\newcommand\eeq{\end{equation}}

\newcommand\ben{\begin{eqnarray}}
\newcommand\een{\end{eqnarray}}

\begin{document}

\title[On additive shifts of multiplicative almost-subgroups] 
{On additive shifts of multiplicative almost-subgroups in finite fields}

\author{Dmitrii Zhelezov}
\thanks{Department of Mathematical Sciences, 
Chalmers University Of Technology and University of Gothenburg} 
\address{Department of Mathematical Sciences, 
Chalmers University Of Technology and University of Gothenburg,
41296 Gothenburg, Sweden} \email{zhelezov@chalmers.se}


\date{\today}

\begin{abstract}
	We prove that for sets $A, B, C \subset \mathbb{F}_p$ with $|A|=|B|=|C| \leq \sqrt{p}$ and a fixed $0 \neq d \in \mathbb{F}_p$ holds 
	$$
		\max(|AB|, |(A+d)C|) \gg|A|^{1+1/26}.
	$$
	In particular, 
	$$
		|A(A+1)| \gg |A|^{1 + 1/26}
	$$
	and
	$$
		\max(|AA|, |(A+1)(A+1)|) \gg |A|^{1 + 1/26}.
	$$
	The first estimate improves the bound by Roche-Newton and Jones.	
	
	In the general case of a field of order $q = p^m$ we obtain similar estimates with the exponent $1+1/559 + o(1)$ under the condition that $AB$ does not have large intersection with any subfield coset, answering a question of Shparlinski.
	
	Finally, we prove the estimate
	$$
		\left| \sum_{x \in \mathbb{F}_q}  \psi(x^n) \right| \ll q^{\frac{7 - 2\delta_2}{8}}n^{\frac{2+2\delta_2}{8}}
	$$
	 for Gauss sums over $\mathbb{F}_q$, where $\psi$ is a non-trivial additive character and $\delta_2 = 1/56 + o(1)$. The estimate gives an improvement over the classical Weil bound when $q^{1/2} \ll n = o\left( q^{29/57 + o(1)} \right)$.
\end{abstract}

\maketitle

\section{Introduction}
One of the main open problems in additive combinatorics is the so-called sum-product phenomenon, which broadly speaking asserts that a set cannot be structured both additively and multiplicatively. There are numerous results quantifying this claim, though we mention just the following results which are of the most relevance to the present note. We will use the standard Vinogradov notation $X \ll Y$ and $X \gg Y$ when there is an absolute constant $C$, independent of $X$ and $Y$, such that $|X| \leq C|Y|$ and $|X| \geq C|Y|$, respectively. Little-oh notation $o(1)$ appears only in exponents and should be regarded as implicit logarithmic factors.

Recall that the multiplicative energy $E^{\times}(A, B)$ is defined as the number of solutions to
$$
	ab = a'b', \,\,\, a, a' \in A, \, b, b' \in B.
$$
If $A = B$ the second argument is usually omitted and the same definition extends to the additive energy $E^{+}$ in the obvious way. We will usually omit the upper index as only additive energies will appear after (\ref{eq:Shkredov}). We refer the reader to  \cite{TaoVu} for details.

Shkredov in \cite{Shkredov} recently proved that if $\Gamma$ is a multiplicative subgroup of $\mathbb{F}_p$ of order $< \sqrt{p}$, then for an arbitrary non-zero residue $x$
\beq \label{eq:Shkredov}
E^{\times}(\Gamma+x) \ll |\Gamma|^2 \log |\Gamma|.
\eeq
In particular, by Cauchy-Schwartz,
$$
\left|  (\Gamma+x) (\Gamma+x) \right| \gg \frac{|\Gamma|^2}{\log |\Gamma|}.
$$

Garaev and Shen \cite{GaraevShen} proved that for a set $A \subset \mathbb{F}_p$ with $|A| \leq \sqrt{p}$
$$
\left| A(A+1) \right|  \gg |A|^{1+\delta}
$$
with $\delta = 1/105 + o(1)$. Roche-Newton and Jones \cite{RNJones} later improved the bound to $\delta = 1/56$. Using an energy estimate from \cite{RudnevRNSh} based on a recent incidence result in finite fields by Rudnev \cite{Rudnev}, we report some progress in a slightly more general setting. Namely we prove the following theorem.

\begin{theorem} \label{thm:main}
   Let $A, B, C \subset \mathbb{F}_p$ with $|A|=|B|=|C| \leq \sqrt{p}$. Then for any fixed $0 \neq d \in \mathbb{F}_p$ holds 
	$$
		\max(|AB|, |(A+d)C|) \gg |A|^{1 + 1/26}.
	$$
	In particular, 
	$$
		|A(A+1)| \gg |A|^{1 + 1/26}
	$$
	and
	$$
		\max(|AA|, |(A+1)(A+1)|) \gg |A|^{1 + 1/26}.
	$$
\end{theorem}
One might think of the theorem above as a generalisation of the aforementioned result of Shkredov to the case of almost-subgroups.

In Section \ref{section:f_q}  we obtain a similar, though worse, bound for the case of a general finite field $\mathbb{F}_q, q=p^m$ thus answering a question of I. Shparlinski.

\begin{theorem} \label{thm:maingeneral}
   Let $A, B, C \subset \mathbb{F}_q$. Suppose that for any proper subfield $F$ and any $c \in \mathbb{F}_q$ holds $|AB \cap cF| \leq |F|^{1/2}$.
   Then for any fixed $0 \neq d \in \mathbb{F}_q$ holds 
	$$
		\max(|AB|, |(A+d)C|) \geq |A|^{1 + 1/559 + o(1)}.
	$$
\end{theorem}

The proof of Theorem \ref{thm:maingeneral} is based on a reduction to the sum-product estimate which was proved by Li and Roche-Newton \cite{LiRN} for general finite fields. In this setting one should obviously require that the set in question does not have large intersection with a subfield coset. The precise statement is as follows. 
\begin{theorem}[Li-Roche-Newton] \label{thm:generalfieldsumproduct}
Let $q = p^n$ and let $A$ be a subset of $\mathbb{F}^{*}_{q}$. If $|A \cap cF| \leq |F|^{1/2}$ for any proper subfield $F$ of $\mathbb{F}_{q}$ and any element $c \in \mathbb{F}_q$, then
$$
	\max\{|A+A|, |AA| \} \gg \frac{|A|^{12/11}}{\log^{5/11}|A|}
$$
\end{theorem}
The exponent in Theorem \ref{thm:maingeneral} is sufficiently worse than in the case of $\mathbb{F}_p$ mainly because of the wasteful use of the Balog-Szemer\'edi-Gowers theorem, so it would be interesting to find a proof which avoids it.

Theorems \ref{thm:main} and \ref{thm:maingeneral} can be used in order to estimate the number of solutions to
$$
g - h = d, \,\,\,\, g \in G, h \in H
$$ 
for a fixed $d \neq 0$ and $G, H$ multiplicative subgroups. Indeed, let $A$ be the set of all $g \in G$ such that $g - h = d$ for some $h \in H$. Then $|AA| \leq |G|$ and $|(A-d)(A-d)| \leq |H|$. 
Applying Theorem \ref{thm:main} we obtain the following corollary.

\begin{corr}
Let $G, H$ be multiplicative subgroups of $\mathbb{F}_p$ with $|G|, |H| \leq \sqrt{p}$ and $0 \neq d \in \mathbb{F}_p$ be fixed. Then the number of solutions to
$$
g - h = d, \,\,\,\,g \in G, h \in H
$$ 
is bounded by 
$$
\max \{ |G|, |H| \}^{26/27}.
$$
\end{corr}
The case of a general field  $\mathbb{F}_{q}$ renders as follows.
\begin{corr} \label{corr:congruences}
Let $G, H$ be multiplicative subgroups of $\mathbb{F}_{q}$ and $0 \neq d \in \mathbb{F}_p$ be fixed. Assume that for any proper subfield $F$ and $c \in  \mathbb{F}_{q}$ holds 
$|G \cap cF| \leq |F|^{1/2}$. Then the number of solutions to
$$
g - h = d, \,\,\,\,g \in G, h \in H
$$ 
is bounded by 
$$
\max \{ |G|, |H| \}^{559/560 + o(1)}.
$$
\end{corr}

In Section \ref{section:GaussSums} we prove a sum-product-type estimate for Gauss sums, providing explicit power-saving exponents for the bound of Bourgain and Chang, \cite{BourgainChang}. 
\begin{theorem} \label{thm:GaussSums}
	Let $q = p^m$ and $n | q - 1$. Assume also that for  any $m \neq \nu | m$ uniformly holds
	\beq \label{eq:n_condition}
		\gcd \left(n, \frac{p^m - 1}{p^\nu - 1} \right) \ll n^\delta \frac{q^{1-\delta}}{p^\nu}	
	\eeq
	with $\delta = 119/605$.
	Then
	$$
		\left| \sum_{x \in \mathbb{F}_q}  \psi(x^n) \right| \ll q^{\frac{7 - 2\delta_2}{8}}n^{\frac{2+2\delta_2}{8}}
	$$
	for any non-trivial additive character $\psi$ and $\delta_2 = 1/56 + o(1)$.
\end{theorem}

\section{Proof of Theorem \ref{thm:main}}
Let $K,L$ be such that $|AB| = K|A|$ and $ |(A+d)C| = L|A|$. We will show that then in fact either $K$ or $L$ are $\gg |A|^{1/26}$.

Let $a_1, a_2, a_3 \in A$, $c \in C$ and $b \in B$. Denote
\ben 
	y_1 &=& (a_1 + d)c \label{eq:defY}\\
	y_2 &=& (a_2 + d)c \nonumber \\
	y_3 &=& (a_3 + d)c, \nonumber	
\een
so $y_i \in (A+d)C$. Our starting point is the following identity, motivated by the incidence result of Helfgott and Rudnev \cite{HelfgotRudnev} and the result of Garaev and Shen \cite{GaraevShen}.
\beq \label{eq:energy}
a_1b - \frac{y_3 - y_1}{y_3 - y_2}a_2b = a_3b\frac{y_1 - y_2}{y_3 - y_2}.
\eeq
 
Here we require that the $y_i$ are pairwise distinct.  Note that if we fix $y_1, y_2, y_3$ then still $b$ can be chosen arbitrarily. Also, if two quadruples $(a_1, a_2, a_3, c)$ and $(a'_1, a'_2, a'_3, c')$ both satisfy (\ref{eq:defY}) with the same $y_i$,  we have
$$
	\frac{y_1}{y_2} = \frac{a_1 + d}{a_2 + d} = \frac{a'_1 + d}{a'_2 + d},
$$
which guarantees that $(a_1b, a_2b) \neq (a'_1b', a'_2b')$ since $d \neq 0$ and $y_1 \neq y_2$. And of course $(a_1b, a_2b) \neq (a_1b', a_2b')$ if $b \neq b'$.

Now we will need some notation in order to estimate the number of triples $(y_1, y_2, y_3)$. Let us reassign  $A' = A+d$ and let $A'(\cdot), C(\cdot)$ be the corresponding indicator functions of  $A'$ and $C$. Write\footnote{Such functions were studied in \cite{ShoenShkredov} with numerous applications.}
$$
\mathcal{C}_4(y_1, y_2, y_3) = \sum_{c \in C} A'(y_1c^{-1})A'(y_2c^{-1})A'(y_3c^{-1}).
$$
Clearly $\mathcal{C}_4$ is supported on $A'C \times A'C \times A'C$.

By the preceding discussion, if the $y_i$'s are pairwise distinct, there are at least $|A|\mathcal{C}_4(y_1, y_2, y_3)$ pairs $(p_1, p_2) \in AB \times AB$ such that
$$
p_1 - \frac{y_3 - y_1}{y_3 - y_2}p_2 \in \frac{y_1 - y_2}{y_3 - y_2}AB.
$$

Denoting $\alpha = \frac{y_3 - y_1}{y_3 - y_2}$ and applying the Cauchy-Schwartz inequality, we have
\beq 
	E(AB, \alpha AB) \geq K^{-1}|A|\mathcal{C}^2_4(y_1, y_2, y_3).
\eeq
or
\beq \label{eq:csenergy}
	E(AB, \alpha AB)^{1/2} \geq K^{-1/2}|A|^{1/2}\mathcal{C}_4(y_1, y_2, y_3).
\eeq
By Corollary 2.10 in \cite{TaoVu}
$$
E(AB, \alpha AB) \leq E(AB, AB)^{1/2}E(\alpha AB, \alpha AB)^{1/2} = E(AB, AB), 
$$
so for we can rewrite (\ref{eq:csenergy}) for the energy $E(AB, AB)$ and then sum over all pairwise distinct triples $(y_1, y_2, y_3) \in A'C \times A'C \times A'C$. The number of such triples is $\gg |A'C|^3 = L^3|A|^3$, so
\beq \label{eq:sumenergy}
	L^3|A|^3E(AB, AB)^{1/2} \gg K^{-1/2}|A|^{1/2} \sum_{(y_1, y_2, y_3) \text{ distinct }} \mathcal{C}_4(y_1, y_2, y_3).
\eeq 
It is easy to see that 
$$
	\sum_{(y_1, y_2, y_3) \in A'C \times A'C \times A'C} \mathcal{C}_4(y_1, y_2, y_3) = |C||A|^3 = |A|^4,
$$
and the contribution from the triples with $y_i = y_j$ is at most $3|C||A|^2 \leq \frac{1}{2}|A|^4$. Thus, (\ref{eq:sumenergy}) gives
$$
L^3|A|^3E(AB, AB)^{1/2} \gg K^{-1/2}|A|^{9/2},  
$$
and
\beq \label{eq:energylowerbound}
	E(AB, AB) \gg L^{-6}K^{-1}|A|^{3} 
\eeq

Now we apply a powerful energy estimate from \cite{RudnevRNSh}. 

\begin{theorem}[Roche-Newton-Rudnev-Shkredov] \label{lm:general_energy}
	Let $X, Y, Z \subset \mathbb{F}_p$, let $M = \max(|X|, |YZ|)$. Suppose that $|X||Y||YZ| \ll p^2$.  Then
$$
	E(X, Z) \ll (|X|||YZ|)^{3/2}|Y|^{-1/2} + M|X||YZ||Y|^{-1}.
$$
\end{theorem}
Before using this theorem, recall the following form of the Pl\"unnecke-Ruzsa inequality due to Ruzsa, see \cite{Ruzsa}.
\begin{lemma} \label{lm:Ruzsa}
	Let $Y, X_1, \ldots, X_k$ be additive subsets of an abelian group. Then
	$$
		|X_1 + \ldots + X_k| \leq \frac{\prod^k_{i=1} |Y + X_i|}{|Y|^{k-1}}.	
	$$
\end{lemma}

The following corollary is due to Katz and Shen, \cite{KatzShen}.
\begin{lemma} \label{lm:KatzShen}
	Let $Y, X_1, \ldots, X_k$ be additive subsets of an abelian group. Then there exists $Y' \subset Y$ with $|Y'| \geq \frac{1}{2}|Y|$ such that
	$$
		|Y' + X_1 + \ldots + X_k| \ll_k \frac{\prod^k_{i=1} |Y + X_i|}{|Y|^{k-1}}.	
	$$
\end{lemma}

We begin with estimating $|AA|$ by Lemma \ref{lm:Ruzsa}. We have 
We have
$$
	|AA| \leq \frac{|AB||AB|}{|B|} = K^2|A|.
$$
The next step is to apply multiplicatively Lemma \ref{lm:KatzShen} to find $Y' \subset A$, $|Y'| \gg |A|$ such that
$$
	|Y'AB| \ll \frac{|AA||AB|}{|A|} \leq K^3|A|.
$$
Now everything is ready to apply Theorem \ref{lm:general_energy} with $X = Z =  AB$ and $Y = Y'$. The hypothesis is satisfied since $|A| \leq \sqrt{p}$ implies $|X||Y||YZ| \ll K^4|A|^3 \ll p^2$ unless $K \gg p^{1/8} \gg |A|^{1/4}$ which is better than the estimate we are aiming at.

We then have 
$$
E(AB, AB) \ll K^{6}|A|^{5/2} + K^{7}|A|^2 \ll K^{6}|A|^{5/2}
$$
provided $K \ll \sqrt{|A|}$ which we can safely assume.

Combining with (\ref{eq:energylowerbound}) we get 
$$
	K^{14}L^{12} \gg |A|
$$
so
$$
	\max \{K, L \} \gg |A|^{1/26},
$$
thus proving  Theorem \ref{thm:main}.

We now record two quick corollaries of our result as claimed in Theorem \ref{thm:main}.
\begin{corr}
	For a subset $A$ of $\mathbb{F}_p$ with $|A| \ll p^{1/2}$ holds
	$$
		|A(A+1)| \geq |A|^{1 + 1/26}.
	$$
\end{corr}
\begin{proof}
	Put $B = A+1$, $C = A$ and $d = 1$.
\end{proof}

\begin{corr}
	For  a subset $A$ of $\mathbb{F}_p$ with $|A| \ll p^{1/2}$ holds
	$$
		\max(|AA|, |(A+1)(A+1)|) \geq |A|^{1 + 1/26}
	$$
\end{corr}
\begin{proof}
	Put $B = A$, $C = A+1$ and $d = 1$.
\end{proof}

\section{Finite fields not of prime order} \label{section:f_q}
In the case of a field not of prime order the energy estimate (\ref{eq:energylowerbound}) still holds since the congruence (\ref{eq:energy}) holds in any field. However, we can no longer use the upper bound of Theorem \ref{lm:general_energy}, so we will use the Balog-Szemer\'edi-Gowers theorem (abbreviates BSG henceforth) instead with the bounds due to Schoen, \cite{Shoen} (see also \cite{TaoVu} and \cite{FoxSudakov}).

\begin{theorem}[Balog-Szemer\'edi-Gowers] \label{thm:BSG}
	If $A$ is a finite subset of an abelian group with the additive energy equal to $\kappa|A|^3$ then there exists $A' \subset A$ with $|A'| \gg \kappa|A|$ such that
	$$
		|A' + A'| \ll \kappa^{-5}|A'|.
	$$
\end{theorem}

If we apply the above theorem to (\ref{eq:energylowerbound}), we obtain a set $X \subset AB$ with $|X| \gg K^{-4}L^{-6}|AB| = K^{-3}L^{-6}|A|$ such that 
$$
|X + X| \ll L^{30}K^{20}|X|.
$$
On the other hand, by the Pl\"unnecke-Ruzsa inequality, 
$$
|XX| \leq |AABB| \leq \frac{|AA||AA||AB||AB|}{|A|^3} \leq K^6|A| \leq L^6K^9|X|.
$$

It is now easy to see that if both $K$ and $L$ are small, then $X$ violates the sum-product inequality of Theorem \ref{thm:generalfieldsumproduct}. We then conclude that if $AB$ has small intersection (i.e. less than $|G|^{1/2}$) with any subfield coset $cG$ then
$$
L^{30}K^{20} = \max\{L^{30}K^{20}, L^6K^9\} \gg |X|^{1/11 + o(1)} \gg K^{-3/11}L^{-6/11}|A|^{1/11+o(1)}, 
$$
which implies
$$
\max \{L, K \} \gg |A|^{1/559+o(1)}.
$$

\section{Gauss sums} \label{section:GaussSums}
Our proof will be based on the estimate for the number for solutions to 
$$
g_1 + g_2 = a
$$ 
for a fixed $a \in \mathbb{F}_q^{*}$ and where $g_1, g_2$ are elements of a multiplicative subgroup $G$. One can readily use Corollary \ref{corr:congruences} but it turns out that a direct application of the Balog-Szemer\'edi-Gowers with a sum-product estimate gives a better bound. 

The central lemma is as follows.
\begin{lemma} \label{lemma:group_energy}
	Let $G \subset \mathbb{F}_q$ be a multiplicative subgroup and $\delta_1 = 486/605$. Assume that for any proper subfield $F$ holds
	\beq \label{eq:field_intersection}
		|G \cap F| \ll |G|^{\delta_1}.	
	\eeq
	Then 
	$$
		E(G) \ll |G|^{3 - \delta_2},	
	$$
	with $\delta_2 = 1/56 + o(1)$.	
\end{lemma}
\begin{proof}
	Let $E(G) = |G|^{3 - \delta'}$ and assume $\delta' < \delta_2 = 1/56 + o(1)$ for contradiction. Let $A$ be the set given by the BSG theorem applied additively to $G$, so 
	$|A| \gg |G|^{1-\delta'}$ and  
	$$
	|A + A| \ll  |A|^{1+\frac{5\delta'}{1-\delta'}}. 
	$$
	The product set $AA$ is clearly contained in $G$, so $|AA| \ll |A|^{1/(1-\delta')}$. 	It remains to check the necessary conditions and apply the sum-product estimate. In fact, we will use a different condition than the hypothesis of Theorem \ref{thm:generalfieldsumproduct}. Namely, one can rule out the case of $A$ having large intersection with a subfield coset (Case 4 in the proof of Theorem 5 in \cite{LiRN})\footnote{More concretely, using the notation of the original paper, one should use (7) of \cite{LiRN}  together with the trivial $|\mathbb{F}_{\tilde{A}_{x_0}} \cap A| \geq \tilde{A}_{x_0}$. } if for any proper subfield coset $cF$ holds
	\beq \label{eq:subfield_condition}
	|cF \cap A| \ll |A|^{1 - 2/11 + o(1)}.  
	\eeq
	Since $|cF \cap A| \leq |F \cap G| \leq |G|^{\delta_1} \leq |A|^{\delta_1/(1-\delta')}$, the condition (\ref{eq:subfield_condition}) holds if $\delta' < \delta_2$. Thus, we can apply the sum-product estimate of Theorem \ref{thm:generalfieldsumproduct} which gives
	$$
		\max \{ |A|^{1+\frac{5\delta'}{1-\delta'}}, |A|^{1/(1-\delta')} \} \gg |A|^{1 + 1/11 + o(1)},
	$$	
	so $1+5\delta'/(1-\delta') = \max\{1+5\delta'/(1-\delta'),  1/(1-\delta')\} \geq 1 + 1/11 + o(1)$, contradicting the original assumption  $\delta' < \delta_2 = 1/56 + o(1)$.
\end{proof}
\begin{remark}
	In the above lemma one can take  $\delta_1$ arbitrarily close to $1$ at the expense of worsening $\delta_2$. 
\end{remark}

A Gauss sum is a sum of the form 
$$
	S_n(a) = \sum_{x \in \mathbb{F}_q} \psi_a(x^n),
$$
where $\psi_a := e_p(\mathrm{Tr}(ax))$ with the standard definitions $e_p(x) := \exp(\frac{2\pi i x}{p})$ and $\mathrm{Tr}(x) = x + x^p + \ldots + x^{p^{m-1}}$. If $a \neq 0$ the character $\psi_a$ is non-trivial and we will assume that henceforth.

The classical Weil bound gives the estimate
$$
	\left| S_n(a) \right| \leq (n-1)q^{1/2},
$$
which becomes trivial if $n \gg q^{1/2}$, so it is interesting to extend the range of $n$ where a non-trivial bound holds. The first result of this kind for fields not of prime order was obtained by Shparlinski \cite{Shparlinski} with the values of $n$ up to $p^{1/6}q^{1/2}$. It was later extended by Bourgain and Chang \cite{BourgainChang}, who proved  
$$
|S_n(a)| \ll q^{1-\delta_\epsilon}
$$
with some unspecified $\delta_\epsilon$ depending only on a fixed parameter $\epsilon$ which controls the interaction with subfields (in particular, $n \leq q^{1-\epsilon}$), similarly to the condition (\ref{eq:n_condition}) of Theorem \ref{thm:GaussSums}. We invite the interested reader to consult \cite{BourgainChang} and \cite{BourgainChang2} for further details. Here our goal is to give an explicit power-saving exponent for $n \gg q^{1/2}$. 

We will use  the following  adaptation of an estimate appeared in \cite{Konyagin} (Lemma 3 with $m = l = 2$).
\begin{lemma}  \label{lm:KonyaginLemma}
	Let $G$ be a mutiplicative subgroup of $\mathbb{F}^*_q$. Then
	\beq \label{eq:KonyaginBound}
	|S(a, G)| := \left|\sum_{g \in G} \psi_a(g) \right| \leq q^{1/8}E(G)^{1/4}.
	\eeq
\end{lemma}
Let $G$ be the group of $n$th powers, which is of order $(q-1)/n$. Combining Lemma \ref{lemma:group_energy} with (\ref{eq:KonyaginBound}) one gets
\beq \label{eq:exp_bound}
|S_n(a)| \ll n|S(a, G)| \ll q^{\frac{7 - 2\delta_2}{8}}n^{\frac{2+2\delta_2}{8}},	
\eeq 
provided the subfield intersection hypothesis of Lemma \ref{lemma:group_energy} is satisfied. The bound (\ref{eq:exp_bound}) is non-trivial if $n = o(q^{(1+2\delta_2)/(2+2\delta_2)}) = o(q^{29/57 + o(1)})$. 

Thus, it remains to check that the condition (\ref{eq:n_condition}) of Theorem \ref{thm:GaussSums} implies (\ref{eq:field_intersection}). Let $F$ be a proper subfield of order $p^{\nu}$. By considering a generator of $\mathbb{F}^*_q$ we can easily calculate
$$
|G \cap F| = \frac{\gcd(n, \frac{q-1}{p^\nu - 1})(p^\nu-1)}{n} \ll \gcd(n, \frac{q-1}{p^\nu - 1})\frac{p^\nu}{q}|G|.
$$
Thus, if (\ref{eq:n_condition}) holds with $\delta = 119/605$ then for any proper subfield $F$ 
$$
|G \cap F| \leq |G|^{\delta_1}
$$
with $\delta_1 = 1-\delta = 486/605$.
\begin{remark}
	Like Lemma \ref{lemma:group_energy}, it is possible to prove Theorem \ref{thm:GaussSums} with $\delta$ arbitrarily close to zero at the expense of worsening $\delta_2$.  
\end{remark}


\section{Acknowledgements}
The author is indebted to Igor Shparlinski for essentially showing how Lemma \ref{lemma:group_energy} applies to Gauss sums. The author would like to thank Alisa Sedunova for bringing the problem (originally due to Igor Shparlinski) to the author's attention. The author is also grateful to Ilya Shkredov, Oliver Roche-Newton and Peter Hegarty for feedback and careful proofreading.


\begin{thebibliography}{99}
   \bibitem{BourgainChang} J. Bourgain and M.-C. Chang, A Gauss sum estimate in arbitrary finite fields, C. R. Acad. Sci. Paris, Ser. I 342 (2006) 643--646.
   \bibitem{BourgainChang2} J. Bourgain and M.-C. Chang, Exponential sum estimates over subgroups and almost subgroups of $\mathbb{Z}^∗_Q$, where $Q$ is composite with few prime factors.  
Geom. Funct. Anal. 16 (2006), no. 2, 327--366. 
    \bibitem{FoxSudakov} J. Fox and B. Sudakov, Dependent Random Choice, arXiv:math/0909.3271 (2009), 31pp.
	\bibitem{GaraevShen} M. Garaev, C.-Y. Shen, On the size of the set $A(A+1)$, Math. Z. 265 (2010), 125--132.
	\bibitem{HelfgotRudnev} H. Helfgott and M. Rudnev, An explicit incidence theorem in $\mathbb{F}_p$, Mathematika 57 (2011), no. 11, 35--145.
	\bibitem{RNJones}  T. Jones and O. Roche-Newton, Improved bounds on the set $A(A+1)$, J. Combin. Theory Ser. A 120 (2013), no. 3, 515--526. 
	\bibitem{KatzShen} N. H. Katz and C.-Y. Shen, A slight improvemnt of Garaev's sum-product estimate, Proc. Amer. Math. Soc. 136(7), 2499--2504. 
	\bibitem{Konyagin} S. V. Konyagin, Estimates of trigonometric sums over subgroups and of Gauss sums, Proceedings of the 4th International Conference "Current problems of number theory and its applications", Tula 2001, Moscow State University Publishing House, Moscow, 2002, 86--114 (in Russian).	
	\bibitem{LiRN} L. Li and O. Roche-Newton, An improved sum-product estimate for general finite fields, SIAM J. Discrete Math. 25 (2011), no. 3, 1285--1296. 	
	\bibitem{RudnevRNSh} O. Roche-Newton, M. Rudnev and I. D. Shkredov,  New sum-product type estimates over finite fields, arXiv:1408.0542.	
	\bibitem{Rudnev} M. Rudnev, On the number of incidences between planes and points in three dimensions, arXiv:1407.0426.
	\bibitem{Ruzsa} M. Ruzsa, An application of graph theory to additive number theory, Scientia, Ser. A 3 (1989), 97--109.
	\bibitem{Shparlinski} I. Shparlinski, Bounds of Gauss sums in finite fields, Proc. Amer. Math. Soc. 132 (2004), no. 10, 2817--2824. 
	\bibitem{Shkredov} I. D. Shkredov, On tripling constant of multiplicative subgroups, arXiv:1504.04522.
	\bibitem{Shoen} T. Schoen, New bounds in Balog-Szemer\'edi-Gowers theorem, Combinatorica, DOI: 10.1007/s00493-014-3077-4, (2014), 1--7 
	\bibitem{ShoenShkredov} T. Schoen and I. D. Shkredov, Higher moments of convolutions, J. Number Theory 133 (2013), no. 5, 1693--1737. 
	\bibitem{TaoVu} T. Tao and V. Vu, Additive combinatorics, Cambridge Studies in Mathematics, Cambridge Univ. Press, 2006.
\end{thebibliography}
\end{document}